\documentclass{amsart}%
\usepackage{amssymb}
\usepackage{amsmath}
\usepackage{latexsym}
\usepackage{hyperref}
\usepackage{amsfonts}
\usepackage{graphicx}%
\newtheorem{theorem}{Theorem}[section]

\newtheorem{remark}[theorem]{Remark}
\newtheorem{proposition}[theorem]{Proposition}
\newtheorem{question}[theorem]{Question}

\newcommand{\be}{\begin{equation}}
\newcommand{\ee}{\end{equation}}
\newcommand{\bea}{\begin{eqnarray}}
\newcommand{\eea}{\end{eqnarray}}

\begin{document}
\title{On a problem of Yau regarding a higher dimensional generalization of the
Cohn-Vossen inequality}
\author{Bo Yang}
\address{Department of Mathematics, University of California San Diego, La Jolla, CA 92093}
\email{{b5yang@math.ucsd.edu}}

\begin{abstract}
We show that a problem by Yau in \cite{Yau} can not be true in
general. The counterexamples are constructed based on the recent
work of Wu and Zheng \cite{WZ}.
\end{abstract}

\maketitle

\section{Introduction}

Shing-Tung Yau asked the following question in \cite{Yau}:

\begin{question} \label{problem 1.1} Given an n-dimensional complete manifold
with nonnegative Ricci curvature, let $B(r)$ be the geodesic ball
around some point $p$. Let $\sigma_{k}$ be the $k$-th elementary symmetric
function of the Ricci tensor. Then is it true that $r^{-n+2k}
\int_{B(r)} \sigma_{k}$ has an upper bound when $r$ tends to
infinity? This should be considered as a generalization of the
Cohn-Vossen inequality.
\end{question}

\vskip 1mm

In the K\"{a}hler category one would like to ask the following
similar question.

\begin{question} \label{Kahler category}  On a complete K\"{a}hler manifold
with complex dimension $n$, if we denote $\omega$ and Ric the
K\"{a}hler form and the Ricci form respectively, one would like to
ask if $r^{-2n+2k} \int_{B(r)} Ric^{k} \wedge \omega^{n-k}$ is
bounded for any $1 \leq k \leq n$ when $r$ goes to infinity.
\end{question}

In this note we exhibit counterexamples to Question \ref{problem
1.1} in the case of $1<k \leq n$ via the recent interesting work of
Wu and Zheng \cite{WZ}. We will show that for any complex dimension
$n \geq 2$ and any $2 \leq k < n$, there exists a $U(n)$ invariant
complete K\"{a}hler metrics on $\mathbb{C}^{n}$ with nonnegative
bisectional curvature such that $r^{-2n+2k} \int_{B(r)} \sigma_{k}$
is unbounded when $r$ large (See Theorem \ref{question 1.1 not
true}). We also prove that Question \ref{Kahler category} is true
for all $U(n)$ invariant complete K\"{a}hler metrics on
$\mathbb{C}^{n}$ with nonnegative bisectional curvature (See Theorem
\ref{question 1.2 true}).

\markright{On a problem of Yau}

\section{Results of Wu and Zheng}

In the section, we collect some of the results from the recent work of
Wu and Zheng \cite{WZ} since they will be used in our constructions of
counterexamples to Question \ref{problem 1.1}. Unless stated otherwise all
results in this section are due to Wu and Zheng \cite{WZ}.

Wu and Zheng \cite{WZ} develops a systematic way
to construct $U(n)$ invariant complete K\"{a}hler metrics on
$\mathbb{C}^{n}$ with positive bisectional curvature. One of the
motivation behind their work is the uniformization conjecture by Yau
\cite{Yau2}. The conjecture states that a complete noncompact
K\"{a}hler manifold with positive bisectional curvature is
biholomorphic to the complex Euclidean space. See
\cite{ChenTangZhu},\cite{ChauTam1}, \cite{ChauTam2}, \cite{ChauTam3}
and reference therein for some recent progress towards Yau's
uniformization conjecture. See also \cite{K},\cite{Cao1}, and
\cite{Cao2} for some earlier works on the construction of
rotationally symmetric complete K\"{a}hler metrics with positive
curvature on $\mathbb{C}^{n}$.

We follow the notations in \cite{WZ}. Let $z=(z_1,\cdots,z_n)$ be the standard
coordinate on $\mathbb{C}^{n}$ and $r=|z|^{2}$. A $U(n)$ invariant
K\"{a}hler metric on $\mathbb{C}^{n}$ has the K\"{a}hler form
\begin{equation}
\omega=\frac{\sqrt{-1}}{2} \partial \overline{\partial} p(r)
\label{Kahler form}
\end{equation} where $p \in C^{\infty}
[0,+\infty)$. Under the local coordinates, the metric has
components:
\begin{equation}
g_{i\overline{j}}=f(r)\delta_{ij}+f'(r) \overline{z}_i z_j.
\end{equation}
We further denote:
\begin{equation}
f(r)=p'(r),\ \ \ \ h(r)=(rf)'.  \label{def of f and h}
\end{equation}
Then the K\"{a}hler from $\omega$ will give a complete metric if and
only if
\begin{equation}
f>0, \ \  h>0, \ \ \int_{0}^{+\infty}  \frac{\sqrt{h}}{\sqrt{r}}
dr=+\infty. \label{being complete}
\end{equation}

Now if we compute the components of the curvature tensor at
$(z_1,0,\cdots,0)$ under the orthonormal frame $\{
e_1=\frac{1}{\sqrt{h}} \partial_{z_1},e_2=\frac{1}{\sqrt{f}}
\partial_{z_2}, \cdots, e_n=\frac{1}{\sqrt{f}} \partial_{z_n} \}$,
then define $A, B, C$ respectively by:
\begin{equation}
A=R_{1\overline{1}1\overline{1}}=-\frac{1}{h} (\frac{rh'}{h})',\
B=R_{1\overline{1}i\overline{i}}=\frac{f'}{f^2}-\frac{h'}{hf},\
C=R_{i\overline{i}i\overline{i}}=2R_{i\overline{i}j\overline{j}}=-\frac{2f'}{f^2},
\label{ABC one form}
\end{equation} where we assume $2 \leq i \neq j \leq n$,
It is easy to check all other components of curvature tensor are
zero.

Let $\mathcal{M}_{n}$ denote the space of all $U(n)$ invariant
complete K\"{a}hler metrics on $\mathbb{C}^{n}$ with positive
bisectional curvature.

\begin{theorem}[\textbf{Characterization of $\mathcal{M}_{n}$ by the $ABC$
function}] Suppose $n \geq 2$ and h is a smooth positive function
on $[0,+\infty)$ satisfying (\ref{being complete}), then the form
defined by (\ref{Kahler form}) gives a complete K\"{a}hler metric
with positive (nonnegative) bisectional curvature if and only if
$A,B,C$ are positive (nonnegative).
\end{theorem}

If we define another function $\xi \in C^{\infty} [0,+\infty)$ by
\begin{equation}
\xi(r)=-\frac{r h'(r)}{h}, \label{def of xi}
\end{equation} then $h$ determines $\xi$ uniquely. On the other hand,
note that $\xi$ determines $h$ by $h(r)=h(0) e^{\int_{0}^{r}
\frac{\xi(t)}{t} dt}$, hence $\omega$ up to scaling. The following
interesting theorem in \cite{WZ} reveals that the space
$\mathcal{M}_{n}$ is in fact quite large.

\begin{theorem}[\textbf{Characterization of $\mathcal{M}_{n}$ by the $\xi$
function}] Suppose $n\geq 2$ and h is a smooth positive function on
$[0,+\infty)$, then the form defined by (\ref{Kahler form}) gives a
complete K\"{a}hler metric with positive bisectional curvature on
$\mathbb{C}^n$ if and only if $\xi$ defined by (\ref{def of xi})
satisfying \begin{equation} \xi(0)=0,\ \ \ \xi^{\prime}>0,\ \ \ \xi<1.
\end{equation}
\end{theorem}

Fix a metric $\omega$ in $\mathcal{M}_{n}$, the geodesic distance
between the origin and a point $z \in \mathbb{C}^{n}$ is:
\begin{equation}
s=\int_{0}^{r}  \frac{\sqrt{h}}{2\sqrt{r}} dr.  \label{s distance}
\end{equation} where $r=|z|^2$. We denote $B(s)$ the ball in
$\mathbb{C}^n$ centered at the origin and with the radius $s$ with
respect to $\omega$. It is further shown in \cite{WZ} that:
\begin{equation}
\operatorname{Vol} (B(s))=c_n (rf)^n.   \label{vol formula}
\end{equation} where $c_n$ is the Euclidean volume of the Euclidean
unit ball in $\mathbb{C}^n$.

Using Theorem 2.2 Wu and Zheng further proved the following
estimates on volume growth of geodesics ball $B(s)$ and the first
Chern number for metrics in $\mathcal{M}_{n}$. Note that an estimate
on volume growth of geodesics ball in the general case has been
proved by Chen and Zhu \cite{ChenZhu2}.

\begin{proposition}[\textbf{Volume growth estimates for metrics in
$\mathcal{M}_{n}$}] $r f = f(1)+2\sqrt{h(1)} (s-s(1))$ for $r>1$ and
$rf \leq s^2$ for any $r \geq 0$. So there exists a constant $C$
such that:
\begin{equation} C s^{n} \leq \operatorname{Vol}(B(s)) \leq c_n s^{2n}.
\end{equation} for $s$ large enough.
\end{proposition}

\begin{proposition}[\textbf{Bounding the first Chern number for metrics in
$\mathcal{M}_{n}$}] Given any $\omega$ in $\mathcal{M}_{n}$ with $n
\geq 1$, we have
\begin{equation}  \int_{C^{n}} (Ric)^n=c_n (\frac{n
\xi(+\infty)}{\pi})^n \leq c_n (\frac{n}{\pi})^{n}.
\end{equation} while $Vol(B(s))=c_n v^n$.
\end{proposition}

In order to construct more examples and compute the scalar curvature
curvature of metrics in $\mathcal{M}_{n}$ in a more convenient way,
Wu and Zheng \cite{WZ} introduced another function $F$ in the
following way: First we define $x=\sqrt{rh}$ and a nonnegative
function $y$ of $r$ by
\begin{equation}
y(0)=0, \ \ \ \ \ {x^{\prime}}^{2}+{y^{\prime}}^{2}=\frac{h}{4r}, \ \ \ \ y' > 0.  \label{def of y}
\end{equation}
One can check that $x(r)$ is strictly increasing and then we may
define $F(x)$ a function on $[0,x_0)$ by $y=F(x)$, where
\begin{equation}
x_0^2=\lim_{r \rightarrow +\infty} rh=h(1)e^{\int_{1}^{+\infty}
\frac{1-\xi}{r} dr}.   \label{def of x0}
\end{equation}
Extending $F$ to $(-x_0,x_0)$ by letting $F(x)=F(-x)$, one can check
that $F$ is a smooth, even function on $|x|<x_0$. Starting with such
a $F$ satisfying certain conditions, one can recover the metric
$\omega$ in a geometric way. See Section 5 in \cite{WZ} for details.
This result is summarized as the following theorem.

\begin{theorem} [\textbf{Characterization of
$\mathcal{M}_{n}$ by the $F$ function}] Suppose $n\geq 1$, there is
a one to one correspondence of between the set $\mathcal{M}_{n}$ and
the set of $\mathcal{F}$ of smooth, even function $F(x)$ defined on
$(-x_0,x_0)$ satisfying
\begin{equation} F(0)=0,\ \ \ \ F''>0,\ \ \ \ \lim_{x \rightarrow x_0} F(x)=+\infty.
\end{equation}
\end{theorem}

\vskip 1mm

Denote $v=rf$, one can rewrite $s$ and $\operatorname{Vol}(B(s))$ in terms of $F$:
\begin{equation}
s=\int_{0}^{x} \sqrt{1+(F'(\tau))^2} d \tau, \ \ \ \operatorname{Vol}(B(s))=c_n
v^n=c_n (\int_{0}^{x}  2\tau \sqrt{1+(F'(\tau))^2} d\tau)^n.
\label{distance and volume}
\end{equation}

Rewrite $A$, $B$, and $C$ defined in (\ref{ABC one form}) in terms
of F:
\begin{equation}
A=\frac{F' F''}{2x(1+{F'}^2)^2}, \ \ \
B=\frac{x^2}{v^2}-\frac{1}{v\sqrt{1+{F'}^2}},\ \ \
C=\frac{2}{v}-\frac{2x^2}{v^2}.  \label{ABC another form}
\end{equation}

Recall the scalar curvature at the point $z=(z_1,0,\cdots,0)$ is
given by
\begin{equation}
R=A+2(n-1)B+\frac{1}{2}n(n-1)C.  \label{scalar curvature}
\end{equation}

Using (\ref{scalar curvature}),(\ref{ABC another form}),
(\ref{distance and volume}) and a careful integration by parts, Wu
and Zheng \cite{WZ} proved the following relation between average
scalar curvature decay and volume growth of geodesic balls. See also
\cite{ChenZhu2} for a related result on any complete K\"{a}hler
manifold with positive bisectional curvature.

\vskip 1mm

\begin{proposition}
[\textbf{Estimates on average scalar
curvature for metrics in $\mathcal{M}_{n}$}]

Given any K\"{a}hler metric $\omega$ in $\mathcal{M}_{n}$ with $n
\geq 2$, there exists a constant $c>0$ such that \begin{equation}
\frac{1}{c(1+v)} \leq \frac{1}{\operatorname{Vol}(B(s))} \int_{B(s)} R(s) w^n \leq
\frac{c}{1+v}.
\end{equation} while $\operatorname{Vol}(B(s))=c_n v^n$.
\end{proposition}

%

\section{Counterexamples to Question 1.1}

Let $\overline{\mathcal{M}}_{n}$ denote the space of all $U(n)$
invariant complete K\"{a}hler metrics on $\mathbb{C}^{n}$ with
nonnegative bisectional curvature. First we state a generalization
of Theorem 2.2 to the space $\overline{\mathcal{M}}_{n}$.

\begin{proposition} [\textbf{Characterization of $\overline{\mathcal{M}}_{n}$ by
the $\xi$ function}] \label{xi function}  Suppose $n\geq 2$ and h is
a smooth positive function on $[0,+\infty)$, then the form defined
by (\ref{Kahler form}) gives a complete K\"{a}hler metric with
nonnegative bisectional curvature if and only if $\xi$ defined by
(\ref{def of xi}) satisfying
\begin{equation}  \xi(0)=0,\ \ \ \xi' \geq 0,\ \ \ \xi \leq 1.
\end{equation}
\end{proposition}

\begin{proof}[Proof of Proposition \ref{xi function}]

The original proof of Theorem 2.2 due to Wu and Zheng \cite{WZ}
works here. Now we only sketch the necessary part. First from (\ref{def of xi}) we know $\xi(0)=0$.
Note that (\ref{def of xi}) and Theorem 2.1 imply
\begin{equation}
A=\frac{\xi^{\prime}}{h} \geq 0
\end{equation} which leads to $\xi^{\prime} \geq 0$.

To prove $\xi \leq 1$, argument by contradiction as in \cite{WZ}. Assume $\lim_{r \rightarrow +\infty}=b>1$, then take
$\delta_{0}>0$ such that $1+\delta_{0}<b$. It follows that there exists $r_0>0$ with $\xi(r_0) \geq 1+\delta_{0}$. Thus
integrating (\ref{def of xi}) leads to
$h(r)=h(0) \exp{\int_{0}^{r} \frac{\xi}{r} dr} \leq \frac{c}{r^{1+\delta_0}} $ which contradicts to the completeness of the metric (\ref{being complete}).
\end{proof}

It also follows from the original proof of Proposition 2.3 and 2.4 due to Wu and Zheng
that the same conclusion holds for the space
$\overline{\mathcal{M}}_{n}$. Namely, for any metric $\omega$ in
$\overline{\mathcal{M}}_{n}$, $C s^{n} \leq \operatorname{Vol}(B(s)) \leq c_n
s^{2n}$ holds for $s$ sufficiently large. and $\int_{C^{n}} (Ric)^n
\leq c_n (\frac{n}{\pi})^{n}$ is true. We remark here that the
estimate on lower bounds of the volume growth of $B(s)$ here can not
be true for an arbitrary complete noncompact K\"{a}hler manifolds
with nonnegative bisectional curvature. For example, take
$\Sigma_{1} \times \mathbb{CP}^{1} \times \cdots \times
\mathbb{CP}^{1}$ where $\Sigma_{1}$ is a capped cylinder on one end
and $\mathbb{CP}^{1}$ is the complex projective plane with the
standard metric.

Next we state another generalization of Theorem 2.5 to
$\overline{\mathcal{M}}_{n}$.

\vskip 1mm

\begin{theorem} [\textbf{Characterization of $\overline{\mathcal{M}}_{n}$
by the $F$ function}] \label{F function}  Suppose $n\geq 1$, there
is a partition of the set $\overline{\mathcal{M}}_{n} \setminus
{\{g_{e}\}}=S_1 \cup S_2 \cup S_3$ where $g_{e}$ is the standard
Euclidean metric on $\mathbb{C}^n$ such that:

(1) $S_1$ has a one to one correspondence with the set of
$\mathcal{F}$ of smooth, even function $F(x)$ on $(-\infty,+\infty)$
defined above satisfying
\begin{equation} F(0)=F'(0)=0,\ \ \ \ F'' \geq 0,\ \ \ \ F'(\infty)< +\infty, \ \ \  F(\infty) = +\infty.
\end{equation}
$S_1$ consists of nonflat K\"{a}hler metrics in
$\overline{\mathcal{M}}_{n}$ whose geodesic balls have Euclidean
volume growth.

(2) $S_2$ has a one to one correspondence with the set of
$\mathcal{F}$ of smooth, even function $F(x)$ on $(-x_0,x_0)$ (where
$x_0$ is either finite or $+\infty$) satisfying
\begin{equation} F(0)=F'(0)=0,\ \ \ \ F'' \geq 0,\ \ \ \ \ F'(x_0)=F(x_0)=+\infty.
\end{equation}
$S_2$ includes K\"{a}hler metrics in
$\overline{\mathcal{M}}_{n}$ whose geodesic balls
have strictly less than Euclidean volume growth and bisectional
curvatures in the radial direction strictly positive along a
sequence of points in $\mathbb{C}^n$ tending to infinity.

(3) For any metric $\omega \in S_3$, there exists a positive real
number $r_0$ such that $r_0=\inf \{ r: \xi(r)=1 \} $ and a
corresponding positive real number $x_0$ such that there exists a
smooth even function $F(x)$ defined on $(-x_0,x_0)$ such that
\begin{equation} F(0)=F'(0)=0,\ \ \ \ F'' \geq 0,\ \ \ \ \
F'(x_0)=+\infty,\ \ \ \ F(x_0)<\infty,
\end{equation} 
$S_3$ is
the set of metrics with geodesic balls having half Euclidean volume
growth and whose bisectional curvatures in the radial direction
vanish outside a compact set. A standard example in complex dimension 1 is a capped cylinder on one end.
\end{theorem}


\begin{proof}[Proof of Theorem \ref{F function}]

The proof of Theorem \ref{F function} is based on a modification of
Theorem 2.5 due to Wu and Zheng. From Proposition \ref{xi function},
we know for any K\"{a}hler metric in $\overline{\mathcal{M}}_{n}$,
there exists a corresponding $\xi(r)$ on $[0,+\infty)$ with
$\xi(0)=0, \xi' \geq 0$, and $\xi \leq 1$. Denote $r_0=\inf {\{ r:
\xi(r)=1 \}}$.

Recall the definition of $x$ and $y$ in (\ref{def of y}), $x=\sqrt{rh}$ and ${x^{\prime}(r)}^2+{y^{\prime}(r)}^2=\frac{h}{4r}$ with $y(0)=0$ and $y^{\prime} \geq 0$. It is easy to check:
\begin{equation}
\frac{dx}{dr}=(1-\xi) \sqrt{\frac{h}{4r}},   \label{dx wrt dr}
\end{equation} then we know $x(r)$ and $y(r)$ are both nondecreasing with respect to $r$.

\textbf{(Case I)} $r_0=+\infty$. From the definition of $x_0$ in (\ref{def of x0}) and (\ref{dx wrt dr})
we know $x(r)$ is strictly increasing on $[0,+\infty)$, then we can define $F(x)$ by $y=F(x)$ on $x \in (-x_0,x_0)$ after an even extension by letting F(-x)=F(x). It is not hard to see that
\begin{equation}
F(0)=0,\ \ \ F'(x) \geq 0, \ \ \ 1+[F'(x)]^2=\frac{1}{(1-\xi)^2}.  \label{about F}
\end{equation}
Recall that $0 \leq \xi(r) \leq 1$ is nondecreasing on
$(-\infty,+\infty)$, we conclude that $F^{\prime \prime} \geq 0$.

Moreover, (\ref{dx wrt dr}) and (\ref{about F}) implies:
\begin{eqnarray}
\lim_{x \rightarrow x_0} F(x) &=&\int_{0}^{x_0}  \sqrt{\frac{1}{(1-\xi)^2}-1}  dx  \label{F(x0)}\\
&=& \int_{0}^{+\infty}  \sqrt{1-(1-\xi)^2} \sqrt{\frac{h}{4r}} dr   \nonumber \\
&\geq& \sqrt{1-(1-\xi(+\infty))^2} \int_{0}^{+\infty} \sqrt{\frac{h}{4r}} dr.  \nonumber
\end{eqnarray}
Note that the integral in the last step of (\ref{F(x0)}) is distance function (\ref{s distance}). we conclude $F(x_0)=\infty$
if and only if $\xi(+\infty)>0$. Note that the latter condition is satisfied when $\omega$ is nonflat.

We further divide our discussion into two subcases:

\textbf{(Subcase Ia)}  $0<\xi(+\infty)<1$. In this case we have $F^{\prime}$ is bounded on $(-x_0,x_0)$ and $x_0=+\infty$.
Moreover, we will prove that the geodesic balls of $(\mathbb{C}^n, \omega)$ has
Euclidean volume growth. We follow the method of Wu and Zheng (See P528 of \cite{WZ}). Note that (\ref{s distance}),(\ref{vol formula}), $(rf)'(r)=h$ and $(rh)'(r)=h(1-\xi)$, it follows from the
L'Hospital's rule that:
\begin{eqnarray}
\lim_{s \rightarrow +\infty}  \frac{\operatorname{Vol} (B(s))}{s^{2n}} &=& \lim_{r \rightarrow +\infty} \frac{c_n (rf)^n}{s^{2n}}  \label{volume growth} \\
&=&\lim_{r \rightarrow +\infty} c_n  (\frac{ \sqrt{rf} }{s})^{2n} \nonumber \\
&=& c_n  (1-\xi(+\infty))^{4n}  \nonumber
\end{eqnarray}

\textbf{(Subcase Ib)}  $\xi(+\infty)=1$, It follows from the (\ref{volume growth}) that in this case the geodesic balls of $(\mathbb{C}^n, \omega)$ has strictly less than Euclidean volume growth. Since $A=\frac{\xi^{\prime}}{h}$, $\xi(0)=0$, and $\xi(+\infty)=1$ we also know that bisectional
curvatures in the radial direction strictly positive along at least a sequence of points in $\mathbb{C}^n$ tending to infinity.

\vskip 1mm

\textbf{(Case II)} $r_0>0$ is finite. Note that $(rh)'=h(1-\xi)$, we
conclude that $x_0=\lim_{r \rightarrow +\infty} \sqrt{rh}$ is finite
and $x_0^2=r_0 h(r_0)$. This implies that $F(x)$ is well defined on
$(-x_0,x_0)$ with $F(x_0)<+\infty$. Since $A=\frac{\xi^{\prime}}{h}$
we conclude that bisectional curvatures in the radial direction
vanishes outside a compact set in $\mathbb{C}^n$. Next we proceed to
show that the geodesic balls of $(\mathbb{C}^n, \omega)$ has half
Euclidean volume growth. Again the methods follows from Wu and Zheng
(See P528 of \cite{WZ}).

\begin{eqnarray}
\lim_{s \rightarrow +\infty}  \frac{\operatorname{Vol} (B(s))}{s^{n}} &=& \lim_{r \rightarrow +\infty} c_n  (\frac{ rf }{s})^{n}  \label{half volume growth}\\
&=& \lim_{r \rightarrow +\infty} c_n  (2 \sqrt{rh})^{n}  \nonumber \\
&=&  2c_n x_0.  \nonumber
\end{eqnarray}

Denote $S_1$, $S_2$, and $S_3$ the sets of metrics in the above three cases (Subcase Ia), (Subcase Ib), and (Case II) respectively, we have proved Theorem \ref{F function}.

\end{proof}

Next we gives some more explicit description of $S_3$. Given any metric
$\omega$ in $S_3$, $\frac{d (rh)}{dr}=(1-\xi) h$, $h=(rf)'$ and $\xi(r)=1$ when
$r>r_0$,  then:
\begin{equation}  rf |_{r_0}^{r}=\int_{r_0}^{r} \frac{r_0 h(r_0)}{r}
dr,
\end{equation} which further implies:
\begin{equation}  rf=x_0^2 \ln{\frac{r}{r_0}}+r_0 f(r_0).
\label{rf}
\end{equation}

Now we compute A, B, C with (\ref{ABC one form}) in Section 2 when $r \geq r_0$:

\begin{equation}
A=-\frac{1}{h} (\frac{rh'}{h})'=\frac{{\xi}'}{h}=0, \label{A for S3}
\end{equation}
\begin{eqnarray}
B &=&  \frac{f'}{f^2}-\frac{h'}{h f}=\frac{1}{r}
(\frac{(rf)'-f}{f^2}-\frac{r h'}{h f}) \label{B for S3} \\  &=& \frac{1}{r}
(\frac{h}{f^2}-\frac{1-\xi}{f})=\frac{h}{r f^2}=\frac{x_0^2}{r^2
f^2}. \nonumber
\end{eqnarray}

\begin{eqnarray}
C=-\frac{2f'}{f^2}=(-2) \frac{h-f}{r f^2} =2 \frac{r f-rh}{(r
f)^2}=2 \frac{x_0^2 (\ln{\frac{r}{r_0}}-1)+r_0 f(r_0) }{[x_0^2
\ln{\frac{r}{r_0}}+r_0 f(r_0)]^2}  \label{C for S3}
\end{eqnarray}

We also see the distance function for metrics in $S_3$:
\begin{equation}
s(r)=\int_{0}^{r_0} \sqrt{\frac{h}{4r}} dr+ \frac{x_0}{2} \ln{\frac{r}{r_0}}.  \label{s for S3}
\end{equation}

Now one can estimate the average of $A$, $B$, and $C$ inside $B(s)$
from (\ref{scalar curvature}), (\ref{A for S3}),(\ref{B for
S3}),(\ref{C for S3}), (\ref{vol formula}), and (\ref{s for S3}). Namely, if $n \geq 2$, for any metric in $S_3$ there
exists a constant $c$ such that
\begin{equation}
\frac{1}{c \, rf} \leq \frac{1}{\operatorname{Vol}(B(s))}
\int_{B(s)} R \, {\omega}^{n} \leq \frac{c}{rf},
\end{equation} where $\operatorname{Vol}(B(s))=c_n (rf)^n$.

If $\omega$ is a nonflat K\"{a}hler metric in $S_1 \cup S_2$, we see from Theorem
\ref{F function} that $F$ must have $F'(x_0)>0$. Then the formula of
$A$, $B$, and $C$ in terms of $F$ is exactly the same as (\ref{ABC
another form}) derived in \cite{WZ} (See P536 in \cite{WZ}). Follow the proof of Proposition
2.6 in \cite{WZ}, we get the same conclusion. We summarize the above discussion
as the following result. Note that $f$ is defined in (\ref{def of f and h}).

\begin{proposition} \label{average scalar
curvature decay} When $n \geq 2 $, given any non flat metric in $
\overline{\mathcal{M}}_{n}$, there exists a constant $C>0$ such that
\begin{equation} \frac{1}{c(1+rf)} \leq \frac{1}{\operatorname{Vol}(B(s))}
\int_{B(s)} R \,  w^n \leq \frac{c}{1+rf}.
\end{equation} where $\operatorname{Vol}(B(s))=c_n (rf)^n$.
\end{proposition}

Now we state the main theorem of this note.
\begin{theorem} \label{question 1.1 not
true}

Given any $n \geq 2 $, any nonflat K\"{a}hler metric in
$\overline{\mathcal{M}}_{n}$ has $\int_{B(s)} \sigma_{n} {\omega}^n$
unbounded when $s$ goes to infinity. Moreover, if $2 \leq k<n$ one
can construct a complete K\"{a}hler metric $\omega$ from $S_1
\subset \overline{\mathcal{M}}_{n}$ with bounded curvature on
$\mathbb{C}^n$ such that $\frac{1}{s^{2n-2k}} \int_{B(s)} \sigma_k
\, {\omega}^n$ is unbounded when $s$ tends to infinity.

\end{theorem}

\begin{proof}[Proof of Theorem \ref{question 1.1 not
true}]

It follows from (\ref{ABC one form}) that for any metric in
$\overline{\mathcal{M}}_{n}$ we have Ricci curvature at $z$ given
by:
\begin{equation}
\lambda=R_{1\overline{1}}=A+(n-1)B,\
\mu=R_{i\overline{i}}=B+\frac{n}{2}C \ \ \     2 \leq i \leq n.
\end{equation} Note that we are now working on the K\"{a}hler manifolds
$\mathbb{C}^n$ and the Ricci tensor is $J$-invariant where $J$ is
the standard complex structure on $\mathbb{C}^n$. Therefore the
Ricci tensor in the real case has eigenvalue $\lambda$ of
multiplicity 2 and $\mu$ of multiplicity $2n-2$. From now on, let
$\sigma_{k}$ denote the $k$-th elementary symmetric function of the
Ricci curvature tensor.

First note that Question 1.1 are true for any metric $\omega \in
\overline{\mathcal{M}}_{n}$ when $k=1$. Since $\sigma_1=2 R$ where
$R$ is the scalar curvature in the K\"{a}hler case, it follows from
Proposition \ref{average scalar curvature decay} and the upper bound
of the volume growth of $B(s)$ after Proposition \ref{xi function}
that $\frac{1}{s^{2n-2}} \int_{B(s)} R\, {\omega}^n$ is bounded when
$r$ tending to infinity. Therefore, we focus on Question 1.1 in the
case of $2 \leq k  \leq n$, If $2 \leq k \leq n$, $\sigma_{k}$ of
the Ricci tensor is a linear combination of ${\lambda}^{2}
{\mu}^{k-2}$, $\lambda {\mu}^{k-1}$, and ${\mu}^{k}$. To sum up,
$\sigma_{k}$ is a linear combination of three types of quantities:

\vskip 1mm

(Type I) $A^{2} B^{i} C^{j}$, $A B^{1+i} C^{j}$, and $ B^{2+i}
C^{j}$ when $i \geq 0$, $j \geq 0$, and $i+j=k-2$.

\vskip 1mm

(Type II) $A B^{i} C^{j}$ and $B^{1+i} C^{j}$ when $i \geq 0$, $j
\geq 0$, and $i+j=k-1$.

\vskip 1mm

(Type III) $B^{i} C^{j}$ when $i \geq 0$, $j \geq 0$, and $i+j=k$.

\vskip 1mm

We divide the proof of Theorem \ref{question 1.1 not
true} into two cases.

(\textbf{Case I}) If $k=n$, we only need to look at the term $C^n$
contained in $\sigma_{n}$. Recall that if for any non flat
K\"{a}hler metric $\omega$ in $S_1 \cup S_2$, we may assume that
there exists $0<M_1<x_0$ such that $F'(x) \geq C_0$ where
$C_0=F'(M_1)>0$ when $x \geq M_1$. we have the expression of $C$
from (\ref {ABC another form}):

\begin{eqnarray}
C &=& \frac{2v-2x^2}{v^2}  \\
& = & \frac{\int_{0}^{x} 2 \tau (\sqrt{1+{F'(\tau)}^2}-1) d\tau}{v^2}  \nonumber\\
& \geq & \frac{\int_{M_1}^{x} 2 \tau
(\frac{F'(\tau)^2}{\sqrt{1+{F'(\tau)}^2}+1}) d\tau} {v
(\int_{M_1}^{x} 2 \tau \sqrt{1+{F'(\tau)}^2} d\tau +\int_{0}^{M_1} 2
\tau \sqrt{1+{F'(\tau)}^2} d\tau )}.   \nonumber
\end{eqnarray}
Note that we have $1 \leq \frac{F'(x)}{C_0}$ when $x \geq M_1$.
\begin{eqnarray}
C \geq
\frac{\frac{C_0}{1+\sqrt{C_0^2+1}}I(x)}{v(\sqrt{\frac{1}{C_0^2}+1}I(x)+M_1^2
\sqrt{1+C_0^2})} \geq \frac{C_1}{v},
\end{eqnarray} where \begin{equation}
C_1=\frac{\frac{C_0}{1+\sqrt{C_0^2+1}}I(x)}{\sqrt{\frac{1}{C_0^2}+1}I(x)+M_1^2
\sqrt{1+C_0^2}}, \ \ \ \ \ \ \ \ \ I(x)=\int_{M_1}^{x} 2\tau
F'(\tau) d\tau.
\end{equation} Since $I(x)$ goes to $\infty$ and $C_1$ is bounded when $x$ tends to
$x_0$, we conclude that there exists a $C_2$ and $M_2$ such that
when $x > M_2$,
\begin{equation}
C \geq \frac{C_2}{v}.   \label {lower bound C}
\end{equation} We remark that (\ref {lower bound C}) is used in the proof
of Proposition 2.3 in \cite{WZ}.

There exists a constant $C_3$ only depending on $n$ such that:
\begin{eqnarray}
\int_{B(s)}    \sigma_{n}  {\omega}^n &\geq& C_3  \int_{B(s)} C^n
{\omega}^n   \label {unbounded} \\
&=& C_3  c_n \int_{0}^{x}  C^n  d v^n   \nonumber\\
&\geq &  C_3  c_n \int_{v(M_2)}^{v(x)}  (\frac{C_2}{v})^n  d v^n   \nonumber\\
&=& n C_3 c_n C_2^{n} \ln{\frac{v(x)}{v(M_2)}}   \nonumber
\end{eqnarray} Of course (\ref{unbounded}) is unbounded when $x$ tends to $x_0$ since $v(x_0)=+\infty$.

To sum up, we show that for any non flat K\"{a}hler metric $\omega$
in $S_1 \cup S_2$, $\int_{\mathbb{C}^n}  \sigma_{n}  {\omega}^n$ is
$\infty$. If $\omega \in S_3$, it follows from (\ref{rf}) and
(\ref{C for S3}) that $\lim_{s \rightarrow +\infty } \int_{B(s)}
\sigma_{n} {\omega}^n$ is unbounded when $s$ goes to infinity.
Therefore, Question 1.1 is false when $k=n$ for any non flat
K\"{a}hler metric $\omega$ in $\overline{\mathcal{M}}_{n}$.

\vskip 1mm

(\textbf{Case II}) If $2 \leq k < n$, For any fixed nonflat
K\"{a}hler metric $\omega$ in $S_1$. we may assume that there exist
$C_4$ and $M_3$ such that $F'(x) \leq C_4$ for all $x \in
(-x_0,x_0)$ and $F'(x) \geq \frac{1}{C_4}$ when $x \geq M_3$. Then
it follows from a similar argument as in (\ref {lower bound C}), we
may further assume that there exist $C_5$ and $M_3$ such that for
any $x \geq M_3$ \begin{equation} C \geq \frac{C_5}{v}.
\end{equation}
 Since
$A=\frac{F' F''}{2x(1+{F'}^2)^2}$, we conclude that $A$ and $
\frac{F''(x)}{x}$ are equivalent. If we can construct an K\"{a}hler
metric $\omega$ in $S_1$ such that
\begin{equation}
\frac{1}{s^{2n-2k}} \int_{B(s)} (\frac{F''(x)}{x})^{2} (\frac{C_5}{v})^{k-2} {\omega}^n
\label{integral for AC}
\end{equation}
is unbounded when $s$ tends to $\infty$, then so is
$\frac{1}{s^{2n-2k}} \int_{B(s)} A^{2} C^{k-2} {\omega}^n$. Note
that the term $A^{2} C^{k-2}$ is contained in $\sigma_k$, it follows
that $\frac{1}{s^{2n-2k}} \int_{B(s)} \sigma_k {\omega}^n$ will be
unbounded when $s$ tends to $\infty$.

Let us rewrite (\ref{integral for AC}):
\begin{eqnarray}
& & \int_{B(s)} (\frac{F''(x)}{x})^{2} (\frac{C_5}{v})^{k-2} {\omega}^n \label{an integral} \\
&=& c_n C_5^{k-2}   \int_{0}^{x}   (\frac{F''(\tau)}{\tau})^{2} (\frac{1}{v})^{k-2} d v^{n} \nonumber\\
&=&  2nc_n C_5^{k-2}  \int_{0}^{x}  \frac{1}{\tau} (F''(\tau))^2
v^{n-k+1} \sqrt{1+(F'(\tau))^2} d\tau. \nonumber
\end{eqnarray}

Since $s=\int_{0}^{x}  \sqrt{1+(F'(\tau))^2} d\tau$, $v=\int_{0}^{x}
2\tau \sqrt{1+(F'(\tau))^2} d\tau$ and $F'(x) \leq C_4$ we know $s$
and $x$ are equivalent, $v$ and $x^2$ are equivalent. In order to
estimate (\ref{an integral}), it suffices to estimate the following.
\begin{equation}
\int_{0}^{x}  (F'')^2 {\tau}^{2(n-k)+1} \, d\tau.
\end{equation}

To sum up, if there exists a function $\delta(x) \in C^{\infty}
[0,+\infty)$, such that
\begin{equation}
\lim_{x \rightarrow x_0} \frac{1}{x^{2n-2k}} \int_{0}^{x} {\delta}^2
(\tau) \, {\tau}^{2(n-k)+1} \, d\tau=+\infty,\ \ \ \ \
\int_{0}^{+\infty} \delta (x) \, dx<+\infty. \label{looking for}
\end{equation} Then we can solve $F(x)$ with $F''(x)=\delta (x)$ with the initial
value $F(0)=F'(0)=0$, it will follow from Theorem \ref{F function} that
we can construct a complete K\"{a}hler metric $\omega$ in $S_1$ such
that
\begin{equation} \frac{1}{s^{2n-2k}} \int_{B(s)} \sigma_k \,
{\omega}^n
\end{equation} is unbounded when $s$ tends to infinity. Hence both
Question 1.1 and 1.2 can not be true when $2 \leq k<n$.

In fact such a $\delta(x)$ is not hard to construct. Consider
$\bar{\delta}(x)$ defined by the following with $q$ an integer to be determined.
\begin {equation}
\bar{\delta}=
\begin{cases}
2    &     x \in [2,2+(\frac{1}{2})^q]   \\
\vdots      &    \vdots    \\
l    &     x \in [l,l+(\frac{1}{l})^q]   \\
\vdots      &    \vdots    \\
0    &     x \in [0,+\infty) \setminus (\cup_{l \geq 2}
[l,l+(\frac{1}{l})^q] )
\end{cases}
\end{equation}

Now set $q=\frac{5}{2}$, it is easy to verify that $\psi(x)$
satisfies (\ref{looking for}). Choose $\delta(x)$ as a suitable
smoothing of $\bar{\delta}(x)$ on $[0,+\infty)$ which also satisfies
(\ref{looking for}), we will get the desired counterexample. It can
be checked that the result metric $\omega \in S_1$ has bounded
curvature on $\mathbb{C}^n$.

\end{proof}

It follows from Theorem \ref{F function} and Proposition
\ref{average scalar curvature decay} that for any K\"{a}hler
metric $\omega \in S_1$, $(\mathbb{C}^n, \omega)$ has quadratic
average scalar curvature decay. Note that the same result for any
complete K\"{a}hler manifolds with bounded nonnegative bisectional
curvature and Euclidean volume growth has been proved by Ni (See
\cite{Ni2} and \cite{Ni3}). Now we construct the following example
which implies that in general only assuming Euclidean volume
growth one can not expect the same rate of decay for
$L^{p}$ norm of curvature for any $p>1$.

\begin{proposition} \label{lp example}  For any $n
\geq 2$ and any $p>1$, there exists a metric $\omega \in
\overline{\mathcal{M}}_{n}$ such that the geodesic balls in
$(\mathbb{C}^n, \omega)$ has Euclidean volume growth. Moreover,
\begin{equation} \frac{s^2}{\operatorname{Vol}(B(s))} \int_{B(s)}
(Rm(\frac{\partial}{\partial s}, J\frac{\partial}{\partial
s},\frac{\partial}{\partial s} J \frac{\partial}{\partial s}))^p
{\omega}^n
\end{equation} is unbounded as $s$ goes to infinity. Here we denote $\frac{\partial}{\partial s}$ to the
unit radial direction on $\mathbb{C}^n$.
\end{proposition}

\begin{proof}[Proof of Proposition \ref{lp example}]
For a given metric $\omega$ in $S_1$, follow a similar argument in (Case II) of the proof of
Theorem \ref{question 1.1 not true}, it suffices to show that we can find a smooth
function $\eta(x)$ on $[0,+\infty)$ such that
\begin{equation}
\lim_{x \rightarrow +\infty} \frac{1}{x^{2n-2}} \int_{0}^{x}
{\eta}^p (\tau) \, {\tau}^{2n-1-p} \, d\tau=+\infty,\ \ \ \ \
\int_{0}^{+\infty} \eta (x) \, dx<+\infty. \label{delta wanted}
\end{equation}

Consider $\bar{\eta}(x)$ defined by the following where $\alpha$ and $\beta$ are two
integers to be determined.
\begin {equation}
\bar{\eta}=
\begin{cases}
2^{\alpha}    &     x \in [2,2+(\frac{1}{2})^{\beta}]   \\
\vdots      &    \vdots    \\
l^{\alpha}    &     x \in [l,l+(\frac{1}{l})^{\beta}]   \\
\vdots      &    \vdots    \\
0    &     x \in [0,+\infty) \setminus (\cup_{l \geq 2}
[l,l+(\frac{1}{l})^q] )
\end{cases}
\end{equation}

Pick any $\alpha>1$ and $1+\alpha < \beta < p(\alpha-1)+2$, then
$\bar{\eta}$ defined above satisfies (\ref{delta wanted}). It is not hard to find
$\eta(x)$ from a suitable smoothing of $\bar{\eta}(x)$ which will result in
the desired metric $\omega$. Note that $(\mathbb{C}^n,\omega)$ we constructed has unbounded curvature
on $\mathbb{C}^n$.   \end{proof}

We proceed to show that Question \ref{Kahler category} is true for
$\overline{\mathcal{M}}_{n}$. It seems that Question \ref{Kahler
category} should be a more suitable conjecture at least for complete
K\"{a}hler manifolds with nonnegative bisectional curvature.

\begin{theorem}  \label{question 1.2 true}  For any
metric $\omega \in \overline{\mathcal{M}}_{n}$, then $s^{-2n+2k} \int_{B(s)} Ric^{k} \wedge
\omega^{n-k}$ is bounded when $s$ goes to infinity.
\end{theorem}

\begin{proof}[Proof of Proposition \ref{question 1.2 true}] \, First we remark that
it directly follows from analogues of Proposition 2.3, 2.4 and 2.6
in \cite{WZ} for the space $\overline{\mathcal{M}}_{n}$ (See
Proposition \ref{average scalar curvature decay} and the paragraph
after Proposition \ref{xi function}) that Question 1.3 is true for
$k=1$ and $k=n$. If suffices to show that $s^{-2n+2k} \int_{B(s)}
Ric^{k} \wedge \omega^{n-k}$ is bounded for any $2 \leq k <n$.

Note that for $2 \leq k <n$, $Ric^{k} \wedge \omega^{n-k}$ is a
linear combination of $\lambda {\mu}^{k-1}$ and ${\mu}^{k}$. It
turns out that we only needs to show that $\frac{1}{s^{2n-2k}}
\int_{B(s)} P(A,B,C) {\omega}^n$ is bounded when $s$ goes to
infinity where $P$ is a monomial of the following two types:

\vskip 1mm

(Type I) $A B^{i} C^{j}$, and $ B^{1+i} C^{j}$ when $i \geq 0$, $j
\geq 0$, and $i+j=k-1$.

\vskip 1mm

(Type II) $B^{p} C^{q}$ when $p \geq 0$, $q \geq 0$, and $p+q=k$.

\vskip 1mm

First we consider any K\"{a}hler metric $\omega$ in $S_1 \cup S_2$.
Note that (\ref{ABC another form}) implies that
\begin{equation}
B \leq \frac{x^2}{v^2} \leq \frac{1}{v},\ \ \ C\leq \frac{2}{v}.
\end{equation}
Then we have the following estimate:
\begin{eqnarray}
& & \frac{1}{s^{2n-2k}} \int_{B(s)} B^{p} C^{q} {\omega}^n \label{hoho1} \\
&\leq & 2^{q} c_n \frac{1}{s^{2n-2k}}  \int_{0}^{v(x)} \frac{1}{v^k}
n v^{n-1} dv \nonumber \\
& \leq & \frac{2^q n c_n  (\int_{0}^{x} 2\tau \sqrt{1+(F'(\tau))^2}
d\tau)^{n-k}}{ (n-k)(\int_{0}^{x}  \sqrt{1+(F'(\tau))^2} d\tau
)^{2n-2k}} \nonumber
\end{eqnarray}
According to the L'Hospital's rule, (\ref{hoho1}) has the limit when
$x$ tends to $x_0$: \begin{equation} \lim_{x \rightarrow
x_0}\frac{(\int_{0}^{x} 2\tau \sqrt{1+(F'(\tau))^2}
d\tau)^{n-k}}{(\int_{0}^{x} \sqrt{1+(F'(\tau))^2} d\tau
)^{2n-2k}}=(\frac{2}{\sqrt{1+\lim_{x \rightarrow x_0}
F'(x)}})^{n-k}.
\end{equation}  We conclude that $\frac{1}{s^{2n-2k}} \int_{B(s)} B^{p} C^{q}
{\omega}^n$ is bounded when $s$ goes to infinity.

Next we turn to the term $A B^{i} C^{j}$, integrate by parts as in
the original proof of Proposition 2.6 in \cite{WZ}.

\begin{eqnarray}
& & \int_{B(s)} A B^{i} C^{j} {\omega}^n \label{hoho2} \\
&=& c_n  \int_{0}^{x} \frac{F'F''}{2\tau (1+(F')^2)^2}
\frac{1}{v^{k-1}}
n v^{n-1} 2\tau \sqrt{1+(F')^2} d\tau  \nonumber  \\
&=& c_n \int_{0}^{x}  \frac{F'F''}{(1+(F')^2)^{\frac{3}{2}}} n v^{n-k} d\tau \nonumber \\
&=&  c_n \int_{0}^{v}  n v^{n-k} d(-\frac{1}{\sqrt{1+(F')^2}}) \nonumber  \\
&=& (-c_n n v^{n-k}\frac{1}{\sqrt{1+(F')^2}})|_{0}^{v}+ c_n
n(n-k)\int_{0}^{v}
\frac{1}{\sqrt{1+(F')^2}} v^{n-k-1} dv  \nonumber \\
&\leq & c_n n v^{n-k}.    \nonumber
\end{eqnarray}
Using the L'Hospital's rule again, we conclude that
\begin{equation}
\frac{1}{s^{2n-2k}} \int_{B(s)} A B^{i} C^{j} {\omega}^n
\end{equation} is bounded when $s$ tends to infinity.

It remains to verify that Question \ref{Kahler category} is true
when $2 \leq k<n$ for any metric $\omega \in S_3$. Note that in this case
we have (\ref{B for S3}), (\ref{C for S3}),  (\ref{s for S3}), and $A=0$ outside a compact set
for metrics in $S_3$, it follows from a
straightforward calculation that $\frac{1}{s^{2n-2k}} \int_{B(s)}
Ric^{k} \wedge \omega^{n-k}$ is bounded when $s$ goes to infinity.
Hence we finish the proof of Proposition \ref{question 1.2 true}.
\end{proof}

\vskip 1mm

We also have the following result relating the growth of the
coordinate function $z_i$ to the volume growth of the geodesic
balls with respect to the metric $\omega$ in
$\overline{\mathcal{M}}_{n}$.

\begin{proposition}  
\label{holomorphic function}

Given any metric $\omega \in \overline{\mathcal{M}}_{n}$, if some
coordinate function $z_i$ for some $1 \leq i \leq n$ has polynomial
growth with respect to $\omega$, then the geodesic balls of
$(\mathbb{C}^n, \omega)$ have Euclidean volume growth.

\end{proposition}

\begin{proof}[Proof of Proposition \ref{holomorphic function}] Assume some
coordinate function $z_i$ for some $1 \leq i \leq n$ has polynomial
growth with respect to $\omega$ in $\overline{\mathcal{M}}_{n}$. it
follows from $\omega$ being rotationally symmetric that there exists
some integer $\alpha$ and constant $C_6>0$ such that:
\begin{equation}
r=|z|^2 \leq C_6 s^{\alpha}.   \label {polynomial growth}
\end{equation}

From Theorem \ref{F function} it suffices to show that $\omega \in S_1$, namely
$F'(x)$ bounded when $x$ goes to $x_0$. First we note that $\omega$
can not be from $S_3$ from the explicit formula (\ref{s for S3}) on the
distance with respect to metrics in $S_3$ given in Theorem 3.1.

Plugging (\ref{distance and volume}) into (\ref{polynomial growth})
leads to:
\begin{equation}
r\leq C_6 (\int_{0}^{x} \sqrt{1+(F'(\tau))^2} d \tau)^{\alpha}.
\label {polynomial growth 2}
\end{equation}

Note that:
\begin{equation}
\frac{dr}{dx}=\frac{2r}{(1-\xi)x}=\frac{2r \sqrt{1+(F'(x))^2}}{x}.
\label{equation r}
\end{equation}
Solve $r$ in terms of $x$ from (\ref{equation r}) and plug into
(\ref {polynomial growth 2}):
\begin{equation}
e^{2\int_{C_7}^{x} \frac{\sqrt{1+(F'(\tau))^2}}{\tau}  d\tau} \leq
C_6 (\int_{0}^{x} \sqrt{1+(F'(\tau))^2} d \tau)^{\alpha}
\label{crucial inequality}
\end{equation} for any $C_7 \leq x < x_0$. Here $C_7$ is the value of $x$
which corresponds to $r=1$.

It is not hard to show $F'(x)$ is bounded for all $x \in (0,x_0)$
from (\ref{crucial inequality}). First we see that $x_0$ must be
infinity. Otherwise, the left hand side $e^{2\int_{C_7}^{x}
\frac{\sqrt{1+(F'(\tau))^2}}{\tau}  d\tau}  \geq  e^{\frac{2}{x_0}
\int_{C_7}^{x} \sqrt{1+(F'(\tau))^2}  d\tau}$. then (\ref{crucial
inequality}) could not be true since exponential functions can not
be controlled by any polynomials when $\int_{C_7}^{x}
\sqrt{1+(F'(\tau))^2} d\tau$ goes to infinity. Next we show that
$F'(x)$ is bounded for all $x \in (0,+\infty)$. It follows from
(\ref{crucial inequality}) that
\begin{equation}
\frac {2\int_{C_7}^{x} \frac{\sqrt{1+(F'(\tau))^2}}{\tau}  d\tau}
{\alpha \ln{\int_{0}^{x}  \sqrt{1+(F'(\tau))^2} d\tau }  +\ln C_6 }
\label {quantity}
\end{equation} should be bounded when $x$ tends to infinity.

It is easy to see that (\ref {quantity}) has a limit when $x$ goes
to infinity.
\begin{equation}
\lim_{x \rightarrow +\infty}  \frac {2\int_{C_7}^{x}
\frac{\sqrt{1+(F'(\tau))^2}}{\tau} d\tau} {\alpha \ln{\int_{0}^{x}
\sqrt{1+(F'(\tau))^2} d\tau }  +\ln C_6 } =\frac{2}{\alpha}
\sqrt{1+(\lim_{x \rightarrow +\infty} F'(\tau))^2},
\end{equation} which implies that $F'(x)$ is bounded for all $x$ in
$[0,+\infty)$. Therefore $\omega \in S_1$. \end{proof}

\vskip 1mm

\begin{remark}
See \cite{Ni1}, \cite{ChenZhu1}, \cite{ChenZhu2}, \cite{Ni2},
\cite{NiTam}, and \cite{Ni3} for some results on the geometry of general
complete noncompact K\"{a}hler manifolds with nonnegative
bisectional curvature. Some further generalizations will also appear
in a separate paper by the author.
\end{remark}

\noindent\textbf{Acknowledgments.} The author thanks Professor Lei
Ni for making Wu and Zheng's paper \cite{WZ} available to him,
bringing Question \ref{Kahler category} to his attention, as well as
many helpful discussions during the preparation of this paper. The
author also thanks Professor Bennett Chow and Professor Xiaohua Zhu
for helpful comments and Professor Fangyang Zheng for his interest
in this work.

\end{document}